\documentclass[10pt,draft]{amsart}
\usepackage{amsmath,amssymb,amsthm}
\begin{document}
\newtheorem{lem}{Lemma}[section]
\newtheorem{prop}{Proposition}[section]
\newtheorem{cor}{Corollary}[section]
\numberwithin{equation}{section}
\newtheorem{thm}{Theorem}[section]

\theoremstyle{remark}
\newtheorem{example}{Example}[section]
\newtheorem*{ack}{Acknowledgments}

\theoremstyle{definition}
\newtheorem{definition}{Definition}[section]

\theoremstyle{remark}
\newtheorem*{notation}{Notation}
\theoremstyle{remark}
\newtheorem{remark}{Remark}[section]

\newenvironment{Abstract}
{\begin{center}\textbf{\footnotesize{Abstract}}%
\end{center} \begin{quote}\begin{footnotesize}}
{\end{footnotesize}\end{quote}\bigskip}
\newenvironment{nome}

{\begin{center}\textbf{{}}%
\end{center} \begin{quote}\end{quote}\bigskip}

\newcommand{\triple}[1]{{|\!|\!|#1|\!|\!|}}

\newcommand{\xx}{\langle x\rangle}
\newcommand{\ep}{\varepsilon}
\newcommand{\al}{\alpha}
\newcommand{\be}{\beta}
\newcommand{\de}{\partial}
\newcommand{\la}{\lambda}
\newcommand{\La}{\Lambda}
\newcommand{\ga}{\gamma}
\newcommand{\del}{\delta}
\newcommand{\Del}{\Delta}
\newcommand{\sig}{\sigma}
\newcommand{\ome}{\omega}
\newcommand{\Ome}{\Omega}
\newcommand{\C}{{\mathbb C}}
\newcommand{\N}{{\mathbb N}}
\newcommand{\Z}{{\mathbb Z}}
\newcommand{\R}{{\mathbb R}}
\newcommand{\T}{{\mathbb T}}
\newcommand{\Rn}{{\mathbb R}^{n}}
\newcommand{\Rnu}{{\mathbb R}^{n+1}_{+}}
\newcommand{\Cn}{{\mathbb C}^{n}}
\newcommand{\spt}{\,\mathrm{supp}\,}
\newcommand{\Lin}{\mathcal{L}}
\newcommand{\SSS}{\mathcal{S}}
\newcommand{\F}{\mathcal{F}}
\newcommand{\xxi}{\langle\xi\rangle}
\newcommand{\eei}{\langle\eta\rangle}
\newcommand{\xei}{\langle\xi-\eta\rangle}
\newcommand{\yy}{\langle y\rangle}
\newcommand{\dint}{\int\!\!\int}
\newcommand{\hatp}{\widehat\psi}
\renewcommand{\Re}{\;\mathrm{Re}\;}
\renewcommand{\Im}{\;\mathrm{Im}\;}

\title
[NLS on product spaces]
{Small data scattering for the nonlinear Schr\"odinger equation on product spaces}
\author[Nikolay Tzvetkov]{Nikolay~Tzvetkov}
\author[Nicola Visciglia]{Nicola Visciglia}
\address{D\'epartement de Math\'ematiques, Universit\'e de Cergy-Pontoise, 2, avenue Adolphe Chauvin, 95302 Cergy-Pontoise  
Cedex, France and Institut Universitaire de France}\email{nikolay.tzvetkov@u-cergy.fr}
\address{Universit\`a Degli Studi di Pisa Dipartimento di Matematica "L. Tonelli"
Largo Bruno Pontecorvo 5 I - 56127 Pisa. Italy}\email{ viscigli@dm.unipi.it}
\maketitle
\begin{abstract}
We consider the cubic nonlinear Schr\"odinger equation, posed on $\R^n\times M$, where $M$ is a compact Riemannian manifold and $n\geq 2$.
We prove that under a suitable smallness in Sobolev spaces condition on the data there exists a unique global solution which scatters to a free solution for 
large times. 
\end{abstract}

\section{Introduction}
The Cauchy problem for the nonlinear Schr\"odinger equation (NLS), posed on a compact Riemannian manifold attracted a considerable attention 
(see in particular \cite{Bo, BGT}).  In all these works the global existence is based on the combination of  (low regularity) 
well-posedness and conservation laws. In such a situation there is
few control on the global dynamics and in particular there is no reason to believe that the solution of the nonlinear problem
is close to the solution of the linear problem for large times, even for small data. In other words scattering is not expected.

On the other hand the Cauchy problem for the nonlinear Schr\"odinger equation, posed on the Euclidean space $\R^n$ is  better understood 
(see for instance \cite{CW, GV,T}).
In particular for small data one expects that the nonlinear evolution is close to the linear one, at least for small data and sufficiently small (near zero) nonlinearity, see for instance \cite{CW}.  An important tool in the proof of such type of results  are the {\it global in time}  Strichartz estimates for the 
linear Schr\"odinger evolution on $\R^n$.  Such type of global in time estimates are false when the problem is posed on a compact manifold.

In view of the previous discussion a natural problem is to consider the NLS on $\R^n\times M$, where $M$ is a compact Riemannian manifold. This is the purpose
of this paper. We will show that the global in time dispersive nature of the $\R^n$ part 
is still sufficient to get small data scattering results similarly to the Euclidean case.
Our view point is to see the problem as a NLS type equation for functions on $\R^n$ with values in 
Sobolev spaces on $M$ (instead of $\C$ for the "usual" NLS). 
We should however admit that our approach, as presented here, is not working for  
problems such as the wave equation on product spaces.

In order to emphasize the main ideas of the paper and to avoid technicalities
we will restrict our attention to the cubic nonlinear interaction, even if 
our approach can be extended to other nonlinearities. Consider thus the Cauchy problem
\begin{equation}\label{mixedcauchy}
{\bf i}\partial_t u +\Delta_{x,y} u=\pm |u|^2u,\quad u(0,x,y)=f(x,y),
\end{equation}
with $(t, x, y)\in \R_t\times \R^n_x\times M_y^k$,
where $M^k_y$ is a compact Riemannian manifold of dimension $k\geq 1$
and $\Delta_{x,y}=\Delta_x + \Delta_y$ with $\Delta_y$
the Laplace-Beltrami operator on $M_y^k$
and $\Delta_x=\sum_{j=1}^n \partial_{x_j}^2$
is the Laplace operator associated to the flat metric on $\R^n$.

We are interested here in the scattering of global solutions to \eqref{mixedcauchy},
under  suitable smallness assumptions on the initial data.
In order to state our first result, we introduce a non isotropic Sobolev space. 
Namely, we denote by 
$\mathcal {H}_{x,y}^{\theta, \rho}$ and the completions of $C^\infty_0(\R^n_x\times M_y^k)$
with respect to to the following norm
$$
\|f\|_{\mathcal {H}_{x,y}^{\theta, \rho}}=
\sum_{|\alpha|\leq \theta}\|\partial_x^\alpha (1-\Delta_y)^{\rho/2} f\|_{L^2(\R^n_x\times M_y)}\,,
$$
where for $\alpha=(\alpha_1,\dots,\alpha_n)\in\N^n$, $\partial^{\alpha}_{x}\equiv\partial_{x_1}^{\alpha_1}\dots \partial_{x_n}^{\alpha_n}$
and $|\alpha|\equiv \alpha_1+\cdots+\alpha_n$. Here is our first result.
\begin{thm}\label{cubic}
Let $n\geq 2$ be even. 
Then for every $\epsilon>0$ there exists $\delta=\delta(\epsilon)>0$
such that the Cauchy problem
\eqref{mixedcauchy} has an unique global solution 
$$u(t,x,y)\in L^{\infty}_t{\mathcal H}^{\frac{n-2}{2}, 
\frac k2+\epsilon}_{x,y} \cap X_\epsilon$$
where
\begin{equation}\label{Xcubiceven}
\|u\|_{X_\epsilon}=\sum_{s=0}^{\frac{n-2}2} \sum_{|\alpha|=s}
\|\partial_x^\alpha (1-\Delta_{y})^{\frac{1}{2}(\frac k2+\epsilon)} u\|_{L^4_t L^{\frac{2n}{1+2s}}_xL^2_y}
\end{equation}
for any initial data $f(x, y)\in {\mathcal H}^{\frac{n-2}{2}, \frac k2+\epsilon}_{x,y}$
such that
$\|f\|_{{\mathcal H}^{\frac{n-2}{2}, \frac k2+\epsilon}_{x,y}}<\delta$.
Moreover 
there exist $f_0^\pm\in {\mathcal H}^{\frac {n-2}2, \frac k2+\epsilon}_{x,y}$
such that
\begin{equation}\label{scattering2}
\lim_{t\rightarrow \pm \infty} \|e^{{\bf i}t\Delta_{x,y}} f_0^\pm 
- u(t,x,y)\|_{{\mathcal H}^{\frac {n-2}2, \frac k2+\epsilon}_{x,y}}=0.
\end{equation}
\end{thm}
\begin{remark}
The eveness of $n$ in Theorem \ref{cubic}
is needed since in this case we are able to estimate 
the cubic nonlinearity in the space $X_\epsilon$ 
by using the usual Leibniz rule.
The case $n$ odd is treated in Theorem \ref{cubicodd} below.
\end{remark}
It is interesting to compare Theorem~\ref{Xcubiceven} in the case $n=k=2$ with the recent result \cite{HTT2}.
In \cite{HTT2}, the problem  \eqref{mixedcauchy}  on $\R^2\times\T^2$ is considered ($\T^2$ is the flat $2d$ torus) and the global well-posedness
for small data in the classical Sobolev spaces 
$H^1(\R^2\times\T^2)$ is proved. The scattering to free solution is not obtained in \cite{HTT2}, the globalization argument being based
on conservation laws together with Tataru's critical spaces theory. Therefore our result says that in the context of the analysis in \cite{HTT2} if in addition one
supposes the smallness of the ${\mathcal H}^{0,1+\epsilon}$ norm then one has scattering. Note that the ${\mathcal H}^{0,1+\epsilon}$ norm is slightly stronger
than the $H^1(\R^2\times\T^2)$ only with respect the $y$ variables. Since in our analysis we do not use any dispersive effect in $y$ it would be interesting to
further understand the interplay between our argument in the case $\R^2\times \T^2$ and the corresponding analysis in $H^1(\R^2\times\T^2)$ in \cite{HTT2}.
It is also worth noticing that our argument here is only restricted to the small data cases while the analysis of \cite{HTT2} also applies to the large data problem,
if we consider sub-cubic defocusing nonlinear interactions. 

We next turn to the odd dimensional case. In this case $(n-2)/2$ is not an integer and a direct application of the proof 
of Theorem~\ref{cubic}  would require some non trivial non isotropic Littlewood-Paley theory.
We decided not to pursue this. Instead, we apply a simple argument which reduces the case of $n\geq 3$ odd to the case of $n$ even.
For $n\geq 3$, we define $\mathcal {H}_{\bar x,(x_n,y)}^{\theta, \rho}$ to be the completion of $C^\infty_0(\R^n_x\times M_y^k)$
with respect to to the following norm
\begin{equation*}
\|f\|_{\mathcal {H}_{\bar x,(x_n,y)}^{\theta, \rho}}=\sum_{|\alpha|\leq \theta}
\|\partial^{\alpha}_{\bar x} 
(1-\partial_{x_n}^2- \Delta_y)^{\rho/2} f\|_{L^2(\R^n_x\times M_y)},
\end{equation*}
where  $\bar{x}=(x_1,\dots,x_{n-1})$ and for 
$\alpha=(\alpha_1,\dots,\alpha_{n-1})\in\N^{n-1}$, $\partial^{\alpha}_{\bar x}\equiv\partial_{x_1}^{\alpha_1}\dots \partial_{x_{n-1}}^{\alpha_{n-1}}$.
Here is our result concerning the odd dimensions $n\geq 3$.
%%%%%%%%%%%%%%%%%%
\begin{thm}\label{cubicodd}
Let $n\geq 3$ be odd. Then for every $\epsilon>0$ there exists  $\delta=\delta(\epsilon)>0$ such that the Cauchy problem
\eqref{mixedcauchy} has an unique global solution 
$$u(t,x,y)\in L^{\infty}_t{\mathcal H}^{\frac{n-3}{2}, 
\frac {(k+1)}2+\epsilon}_{\bar x,(x_n,y)} \cap X_\epsilon$$
where
\begin{equation}\label{Xcubicodd}
\|u\|_{X_\epsilon}=\sum_{s=0}^{\frac{n-3}2} 
\sum_{|\alpha|=s}\|\partial_{\bar x}^\alpha(1-\partial_{x_n}^2- \Delta_y)^{\frac{k+1}4+\epsilon}u\|_{L^4_t L^{\frac{2(n-1)}{1+2s}}_xL^2_{(x_n,y)}}
\end{equation}
for any initial data $f(x, y)\in {\mathcal H}^{\frac{n-3}{2}, 
\frac{k+1}2+\epsilon}_{\bar x,(x_n,y)}$
such that
$\|f\|_{{\mathcal H}^{\frac{n-3}{2}, 
\frac{k+1}2+\epsilon}_{\bar x,(x_n,y)}}<\delta$.
Moreover 
there exist $f_0^\pm\in {\mathcal H}^{\frac {n-3}2, 
\frac{k+1}2+\epsilon}_{\bar x,(x_n,y)}$
such that
\begin{equation}\label{scattering3}\lim_{t\rightarrow \pm \infty} 
\|e^{{\bf i}t\Delta_{x,y}} f_0^\pm 
- u(t,x,y)\|_{{\mathcal H}^{\frac {n-3}2, \frac{k+1}2+\epsilon}_{\bar x,(x_n,y)}}=0.
\end{equation}
\end{thm}
%%%%%%%%%%%%%%%%%%%%%%%%%%%%%%%%%%%%%%%
An analogue of Theorem~\ref{cubicodd} can not hold for $n=1$. In this case we expect a suitable Banach space valued version of the modified scattering result 
of Ozawa \cite{O}. Such a result will give an insight into the global small data dynamics of the cubic NLS on $\R\times\T$, established in \cite{TT}.

In the case $n=1$ one can obtain, by invoking a Banach space valued version of the small data theory of the quintic NLS on $\R$,
an analogue of Theorem~\ref{cubicodd} if the cubic non linearity is replaced by the quintic one, namely $\pm |u|^2u$ replaced by $\pm |u|^4u$.

The remaining part of the paper is organized as follows. In the next section, we establish a basic  Strichartz inequality. 
This inequality only uses the dispersive effect in the $x$ variables but have the advantage to be global in time.
Next, we prove Theorem~\ref{cubic}. The final section is devoted to the proof of Theorem~\ref{cubicodd}.
%%%%%%%%%%%%%%%%%%%%
%%%%%%%%%%%%%%%%%%%%%%%%%%%%%%%%%%%%%%%%%%%%%%%%%%%%%%%%
%%%%%%%%%%%%%%%%%%%%%%%%%%%%%%%%%%%%%%%%%%%%%%%%%%%%%%%%%
\section{A Strichartz type inequality}
In this section, we establish our basic tool which is a Strichartz type estimate for $e^{{\bf i}t\Delta_{x,y}}$. Here is the precise statement.
%%%%%%%%%%
\begin{prop}\label{strichrac}
For every $n\geq 1$ and for every compact Riemannian manifold $M^k_y$
the following estimate holds:
\begin{multline}\label{strichartzeven}
\|e^{{\bf i}t\Delta_{x,y}} f\|_{L^p_tL^q_xL^2_y} + \|\int_0^t e^{{\bf i}(t-\tau)\Delta_{x,y}} F(\tau, x,y) d\tau\|_{L^p_tL^q_xL^2_y}
\\
\leq C (\|f\|_{L^2_{x,y}} 
+ \|F\|_{L^{\tilde p'}_tL^{\tilde q'}_x L^2_y}),
\end{multline}
where $C=C(p, q, \tilde p, \tilde q)>0$ and 
$$\frac 2p + \frac nq=\frac n2,\quad \frac 2{\tilde p} + \frac n{\tilde q}=\frac n2$$
$2\leq p,\tilde p\leq \infty$ for $n>2$, $2<p, \tilde p\leq \infty$ for $n=2$
and $4\leq p, \tilde p\leq \infty$ for $n=1$.
\end{prop}
%%%%%%%%%%%%%%%%%%%%%%%%
\begin{proof}
Let us recall the usual Strichartz estimates 
for the free propagators $e^{{\bf i}t(\Delta_x+m)}$ on $\R^n_x$ with $m\in \R$:
\begin{equation}\label{strichartzfree}
\sup_{m\in \R} (\|e^{{\bf i}t(\Delta_{x}+m)} h\|_{L^p_tL^q_x} 
+ \|\int_0^t e^{{\bf i}(t-\tau)(\Delta_{x}+m)} H(\tau, x) d\tau\|_{L^p_tL^q_x})
\end{equation}
$$\leq C (\|h\|_{L^2_{x}} 
+ \|H\|_{L^{\tilde p'}_tL^{\tilde q'}_x})$$
under the same assumptions on $p,\tilde p, q, \tilde q$,
with $C=C(p, \tilde p, q, \tilde q)>0$ that does not depend on $m$.
Recall that the usual Strichartz estimate concerns the propagator
$e^{{\bf i}t\Delta_x}$. On the other hand in \eqref{strichartzfree}
we are allowed to get uniform bounds with respect to $m\in \R$ since
$e^{{\bf i}t(\Delta_x +m)}=e^{{\bf i}tm}e^{{\bf i}t\Delta_x}$
and moreover the Strichartz norm are not affected
by the remodulation factor $e^{{\bf i}tm}$.
Next we introduce
$$
u(t, x,y)=e^{{\bf i}t\Delta_{x,y}} f +
\int_0^t e^{{\bf i}(t-\tau)\Delta_{x,y}} F(\tau,x,y) d\tau
$$
and notice that
$${\bf i} \partial_t u+ \Delta_{x} u + \Delta_y u = F, \qquad 
(t,x,y)\in \R\times \R_x^n\times M_y$$
with
$$u(0,x,y)= f(x, y).$$
Let us decompose
$$u(t,x, y), f(x, y) \hbox{ and } F(t,x,y) $$ 
with respect to the orthonormal basis $\{\varphi_j(y)\}$ of $L^2(M_y)$
given by the eigenfunctions of $-\Delta_y$ (i.e. 
$-\Delta_y \varphi_j=\lambda_j \varphi_j$)
\begin{equation}\label{sopra}
u(t,x, y)=\sum_{j} u_j(t, x) \varphi_j(y)
\end{equation}
\begin{equation}\label{mezzo}
F(t,x, y)=\sum_{j} F_j(t,x) \varphi_j(y)
\end{equation}
\begin{equation*}
f(x, y)=\sum_{j} f_j(x) \varphi_j(y)
\end{equation*}
and notice that
$u_j(t, x), F_j(t,x) \hbox{ and }f_j(x) $
are related by the following Cauchy problems:
\begin{equation}\label{reduction}
{\bf i} \partial_t u_j+\Delta_{x} u_j - \lambda_j u_j = F_j, (t,x)\in \R_t\times \R^n_x
\end{equation}
with
\begin{equation*}u_j(0,x)= f_j(x).
\end{equation*}
Applying \eqref{strichartzfree} in the context of \eqref{reduction} gives
$$\|u_j(t,x)\|_{L^p_tL^q_x}
\leq C\|f_j\|_{L^2}+C \|F_j(t,x)\|_{L^{\tilde p'}_tL^{\tilde q'}_x}$$
and hence summing in $j$ the squares we get
$$ \|u_j(t, x)\|_{l^2_{j} L^p_tL^q_x}\leq 
C\|f\|_{L^2_{x,y}}+C\|F_j(t, x)\|_{l^2_{j}L^{\tilde p'}_t L^{\tilde q'}_x}.$$
On the other hand
$$
\max\{\tilde p', \tilde q'\}\leq 2\leq \min\{p, q\}
$$
and therefore by the Minkowski inequality we get 
$$ \|u_j(t, x)\|_{L^p_tL^q_x l^2_{j}}\leq C \|f\|_{L^2_{x,y}}+C\|F_j(t, x)\|_{L^{\tilde p'}_t L^{\tilde q'}_x l^2_{j}}.$$
Combining \eqref{sopra} and \eqref{mezzo} with the Plancharel identity gives
$$
\|u\|_{L^p_tL^q_x L^2_{y}}\leq C \|f\|_{L^2_{x,y}}+C\|F\|_{L^{\tilde p'}_t L^{\tilde q'}_x L^2_{y}}.
$$
Finally, we apply the last inequality first with $f=0$ and then $F=0$ to achieve the bound \eqref{strichartzeven}.
This completes the proof of Proposition~\ref{strichrac}.
\end{proof}
%%%%%%%%%%%%%%%%%%%%%%%%%%
\section{Proof of Theorem \ref{cubic}}\label{cubicsec}
We recall a suitable version of Strichartz estimates
for the classical propagator $e^{{\bf i}t(\Delta_x+m)}$ on $\R^n_x$ with $n$ an even integer
and $m\in \R$.
Notice that we use that $n$ is even in order to have that $(n-2)/2$ is an integer
which allows us to give a meaning to the 
derivation operators up to $(n-2)/2$ that appears on the r.h.s.
of \eqref{uniformmstri} (see also below remark \ref{s}).
\begin{prop}\label{above}
Let $n\geq 2$ be even,
$\alpha\in \N^n$ such that $0\leq |\alpha|\leq \frac{n-2}{2}$, then
\begin{multline}\label{uniformmstri}
\sup_{m\in \R} \Big(\|\partial_x^{\alpha} 
e^{{\bf i}t(\Delta_x+m)} h\|_{L^p_tL^q_x}+ \|\partial_x^\alpha (\int_0^t e^{{\bf i}(t-\tau)(\Delta_x+m)} H(\tau,x) d\tau)\|_{L^p_tL^q_x}\Big)
\\
\leq 
C\Big(\|h\|_{H^{\frac{n-2}2}_x}+\sum_{|\beta|= \frac{n-2}{2}}\|\partial^\beta_x H\|_{L^{\tilde p'}_t L^{\tilde q'}_x}\Big),
\end{multline}
where $C=C(p, q, \tilde p, \tilde q)>0$ does not depend on $m\in \R$, 
$$
\frac2p+\frac nq=1+|\alpha|,\quad \frac2{\tilde p}+\frac n{\tilde q}=\frac n2,\quad 2<p, \tilde p\leq \infty.
$$
\end{prop}
\begin{proof}
Observe that, under our restriction on $p$, we have $1<q<\infty$. We have the Sobolev embedding 
$$
\dot{W}^{\frac{n-2}{2}-|\alpha|,q_1}(\R^n)\subset L^q(\R^n),\quad \frac{n}{q_1}-\frac{n}{q}=\frac{n-2}{2}-|\alpha|\implies \frac{2}{p}+\frac{n}{q_1}=\frac{n}{2}.
$$
Therefore the left hand-side of \eqref{uniformmstri} is bounded by
$$
\sup_{m\in \R} \sum_{|\beta|=\frac{n-2}{2}}
\Big(\|\partial_x^{\beta} e^{{\bf i}t(\Delta_x+m)} h
\|_{L^p_tL^{q_1}_x}+ \|\partial_x^\beta (\int_0^t e^{{\bf i}(t-\tau)(\Delta_x+m)} H(\tau,x) d\tau)
\|_{L^p_tL^{q_1}_x}\Big).
$$
Thus the estimate \eqref{uniformmstri}
for a fixed $m$ follows the usual Strichartz estimates thanks to the relation
$\frac{2}{p}+\frac{n}{q_1}=\frac{n}{2}$.
The uniformity of the estimate with respect to $m\in \R$ can be deduced 
as in \eqref{strichartzfree}.
\end{proof}
%%%%%%%%%%%%%%
\begin{remark}\label{s}
Notice that in the case $n$ odd an estimate similar to \eqref{uniformmstri}
is satisfied provided that the local operator $\partial_x^\alpha$ is replaced by $(1-\Delta_x)^{|\alpha|/2}$.
However for our purpose it will be relevant to work with $\partial_x^\alpha$, in view of the possibility to apply for this operator the usual Leibniz rule
for the derivation of a product.
\end{remark}
The next result will be fundamental in the sequel.
%%%%%%%%%%%%%%%%%%%%%%%%%%%
\begin{prop}\label{mixed}
Let $n\geq 2$ be even,
$\alpha\in \N^n$ such that $0\leq |\alpha|\leq \frac{n-2}{2}$ and $r\geq 0$, then we have
\begin{multline}\label{roip}
\|\partial^{\alpha}_x (1-\Delta_y)^{r/2}e^{{\bf i}t\Delta_{x,y}} f\|_{L^p_tL^q_xL^2_y}
\\
+ \|\partial_x^\alpha (1-\Delta_y)^{r/2} (\int_0^t e^{{\bf i}(t-\tau)\Delta_{x,y}} 
F(\tau) d\tau)\|_{L^p_tL^q_xL^2_y}
\\
\leq C \Big(\|f\|_{{\mathcal H}^{\frac{n-2}2,r}_{x,y}}
+\sum_{|\beta|= \frac{n-2}{2}}
\|\partial^{\beta}_x (1-\Delta_y)^{r/2} F
\|_{L^{\tilde p'}_t L^{\tilde q'}_x L^2_y}\Big),
\end{multline}
where $C=C(p, \tilde p, q, \tilde q)>0$,
$$
\frac2p+\frac nq= 1+|\alpha|,\quad
\frac2{\tilde p}+\frac n{\tilde q}=\frac n2,\quad 2<p, \tilde p\leq \infty.
$$
\end{prop}
\begin{proof}
It is sufficient to consider the case $r=0$.
The general case follows simply by a derivation of the equation with respect to
$y$ variables.
The proof of \eqref{roip} for $r=0$ is similar to the proof
of Proposition \ref{strichrac} by using Proposition \ref{above}
instead of \eqref{strichartzfree}.
\end{proof}
In the sequel we shall work with the following norm:
\begin{equation}\label{Xcubicevensec3}
\|u\|_{X_\epsilon}=\sum_{s=0}^{\frac{n-2}2} \sum_{|\alpha|=s}
\|\partial_x^\alpha (1-\Delta_{y})^{\frac{1}{2}(\frac k2+\epsilon)} 
u\|_{L^4_t L^{\frac{2n}{1+2s}}_xL^2_y}.
\end{equation}
%%%%%%%%%%%%%%%%%%%%%%%%%%
\begin{prop}\label{multilinear}
Assume $n$ is even and $\epsilon>0$.  Then
there exists $C=C(n,\epsilon)>0$ such that for every $\beta\in\N^d$ satisfying  $|\beta|= \frac{n-2}{2}$,
every $u_1,u_2,u_3\in X_\epsilon$,
$$
\|\partial_x^{\beta} 
(1-\Delta_y)^{\frac{1}{2}(\frac k2+\epsilon)}
(u_1u_2u_3)\|_{L^{\frac 43}_t L^{\frac{2n}{n+1}}_x L^2_y} \leq C \|u_1\|_{X_\epsilon}\|u_2\|_{X_\epsilon}\|u_3\|_{X_\epsilon}.
$$
\end{prop}
\begin{proof}
Notice that $\mathcal H^{\frac k2+\epsilon}_y$ is an algebra
and hence
\begin{equation}\label{algfin}
\|(1-\Delta_y)^{\frac{1}{2}(\frac k2+\epsilon)}
(f_1f_2f_3)\|_{L^2_y} 
\leq C\prod_{j=1}^{3} \|(1-\Delta_y)^{\frac{1}{2}(\frac k2+\epsilon)} f_j \|_{L^2_y}\,.
\end{equation}
Moreover, by the Leibnitz formula
\begin{equation}\label{leibnitz}
\partial^{\beta}_x (g_1g_2g_3)
=
\sum_{|\beta_1|+|\beta_2|+|\beta_3|=\frac{n-2}{2}}
c_{\beta_1 \beta_2 \beta_3} (\partial_x^{\beta_1} 
g_1\partial_x^{\beta_2} g_2\partial_x^{\beta_3} g_3)
\end{equation}
for a suitable choice of the coefficient $c_{\beta_1\beta_2\beta_3}$.
By combining \eqref{algfin}, \eqref{leibnitz} with the Minkowski inequality
it is sufficient to prove
$$
\left \|\prod_{j=1}^{3} \|(1-\Delta_y)^{\frac{1}{2}(\frac k2+\epsilon)}
\partial_{x}^{\beta_j} u_j\|_{L^2_y}
\right \|_{L^\frac 43_t L^{\frac{2n}{n+1}}_x}
\leq C\|u_1\|_{X_\epsilon}\|u_2\|_{X_\epsilon}\|u_3\|_{X_\epsilon}
$$
with $|\beta_1|+|\beta_2|+|\beta_3|=\frac{n-2}{2}.$
Using the relation
$$\frac{n+1}{2n}=\frac{1+2|\beta_1|}{2n}+\frac{1+2|\beta_2|}{2n}+\frac{1+2|\beta_3|}{2n}$$
and the H\"older inequality, applied with respect to $(t,x)$, we get
\begin{multline*}
\left \|\prod_{j=1}^{3} \|(1-\Delta_y)^{\frac{1}{2}(\frac k2+\epsilon)}
\partial_{x}^{\beta_j} u_j\|_{L^2_y}
\right \|_{L^\frac 43_t L^{\frac{2n}{n+1}}_x}
\leq 
\\
\prod_{j=1}^{3} \Big\|(1-\Delta_y)^{\frac{1}{2}(\frac k2+\epsilon)}
\partial_x^{\beta_j}  u_j\Big \|_{L^\frac 43_t L^{\frac{2n}{2|\beta_j|+1}}_xL^2_{y}}
\leq \|u_1\|_{X_\epsilon}\|u_2\|_{X_\epsilon}\|u_3\|_{X_\epsilon}\,.
\end{multline*}
This completes the proof of Proposition~\ref{multilinear}.
\end{proof}
%%%
{\bf Proof of Theorem \ref{cubic}}
The problem  \eqref{mixedcauchy} can be rewritten as the integral equation
\begin{equation}
u(t)= e^{{\bf i}t\Delta_{x,y}} f\pm \int_0^t e^{{\bf i}(t-s)\Delta_{x,y}}\Big(|u(s)|^2 u(s)\Big) ds\equiv T_{f}(u). 
\end{equation}
The proof of \eqref{scattering2} is standard once it is proved the existence of a global solution
$u(t,x,y)$ belonging to 
$$
Y_\epsilon=L^{\infty}_t{\mathcal H}^{\frac{n-2}{2}, \frac k2+\epsilon}_{x,y} \
\cap X_\epsilon
$$
and hence will be omitted (for more details on this fact see
\cite{cazenave}). By a fixed point argument it is sufficient to prove the following\\
\\
{\em Claim}:
$$\forall\, \epsilon\in (0,\infty) \hbox{ }
\exists\, \delta=\delta(\epsilon)>0 \hbox{ and } R=R(\epsilon)>0 \hbox{ s.t. }
T_f (Y_{\epsilon, R})\subset Y_{\epsilon, R}
$$$$\hbox{ and } T_f \hbox{ is a contraction on } Y_{\epsilon, R}
\hbox{ } \forall f \hbox{  s.t. } \|f\|_{{\mathcal H}^{\frac{n-2}2, 
\frac k2+\epsilon}_{x,y}}<\delta,$$\\
\\
where $Y_{\epsilon,R}=\{u\in Y_{\epsilon}|\|u\|_{Y_\epsilon}<R\}$.\\
By combining \eqref{roip}
with Proposition \ref{multilinear} we get:
$$\|T_f u\|_{Y_\epsilon}
\leq C (\|f\|_{{\mathcal H}^{\frac{n-2}2, 
\frac k2+\epsilon}_{x,y}} + 
\sum_{|\beta|=\frac{n-2}{2}}
\|\partial_x^{\beta}(1-\Delta_y)^{\frac{1}{2}(\frac k2+\epsilon)}
(u|u|^2)\|_{L^\frac 43_t  L^{\frac{2n}{n+1}}_x 
L^2_y})$$
$$\leq C (\|f\|_{{\mathcal H}^{\frac{n-2}2, 
\frac k2+\epsilon}_{x,y}} + \|u\|_{X_\epsilon}^3)$$
$$\leq C (\|f\|_{{\mathcal H}^{\frac{n-2}2, 
\frac k2+\epsilon}_{x,y}} + \|u\|_{Y_\epsilon}^3).$$
By a standard continuity argument the previous estimate
gives the existence of $\delta>0$ and $R(\delta)>0$ such that
$$T_f(Y_{\epsilon, R(\delta)})\subset Y_{\epsilon, R(\delta)}$$ provided that
$\|f\|_{{\mathcal H}^{\frac{n-2}2,\frac k2 +\epsilon}_{x,y}}<\delta$.
Moreover $\lim_{\delta\rightarrow 0} R(\delta)=0$.
Next we shall check that $T_f$ is a contraction on $Y_{\epsilon,R(\delta)}$
provided that $\delta$, and hence $R(\delta)$, are small.
By using 
\eqref{roip}
we get
$$\|T_f(v)-T_f(w)\|_{Y_\epsilon}
\leq C \sum_{|\beta|=\frac{n-2}{2}} \|\partial_x^{\beta}(1-\Delta_y)^{\frac{1}{2}(\frac k2+\epsilon)}(v|v|^2-w|w|^2)
\|_{L^\frac 43_t  L^{\frac{2n}{n+1}}_x 
L^2_y}$$
that in conjunction with the following identity
$$v^2\bar v - w^2\bar w=
(v-w)(v+w)\bar w+ v^2(\bar v-\bar w)$$
and with Proposition \ref{multilinear} gives
$$\|T_f(v)-T_f(w)\|_{Y_\epsilon}
\leq C \|v-w\|_{Y_\epsilon} 
(\|v\|_{Y_\epsilon}+ \|w\|_{Y_\epsilon})^2\leq
C \|v-w\|_{Y_\epsilon} (R(\delta))^2.$$
Hence $T_f$ is contraction on $Y_{\epsilon, R(\delta)}$
in case $R(\delta)$ is small enough.
%%%%%%%%%%%%%%%%%%%%%%%%%%%%%%%%%%%%%%%%%%%%
%%%%%%%%%%%%%%%%%%%%%%%%%%%%%%%%%%%%%%%%%%%%
\section{Proof of Theorem \ref{cubicodd}}\label{cubicoddsec}
Notice that if we split
$$\R^n_x\times M^k_y=\R^{n-1}_{\bar x}\times (\R_{x_n}\times M_y^k)$$
then we are reduced to the situation of Theorem
\ref{cubic} since $n-1$ is an even number.
However we are not allowed to apply directly
Theorem \ref{cubic} since the manifold
$\R_{x_n}\times M_y^k$ is not compact 
(despite to the assumption of Theorem
\ref{cubic}). To overcome this difficulty we shall prove 
the following version of Proposition~\ref{mixed}.
\begin{prop}\label{mixedodd}
Let $n\geq 3$ be odd, $\alpha\in \N^n$ such that $0\leq |\alpha|\leq \frac{n-3}{2}$ 
and $r\geq 0$, then
\begin{multline}\label{cauchyduhamel}
\|\partial^{\alpha}_{\bar x} 
(1-\partial_{x_n}^2-\Delta_{y})^{r/2}
e^{it\Delta_{x,y}} f
\|_{L^p_tL^q_{\bar x}L^2_{(x_n,y)}}
\\
+ \|\partial_{\bar x}^\alpha 
(1-\partial_{x_n}^2-\Delta_{y})^{r/2}
(\int_0^t e^{{\bf i}(t-\tau)\Delta_{x,y}} F(\tau) d\tau)
\|_{L^p_tL^q_{\bar x}L^2_{(x_n,y)}}
\\
\leq C\Big(\|f\|_{{\mathcal H}^{\frac{n-3}2,r}_{\bar x,(x_n,y)}}
+
\sum_{|\beta|=\frac{n-3}{2}}
\|\partial^{\beta}_{\bar x} (1-\partial_{x_n}^2-\Delta_{y})^{r/2} F
\|_{L^{\tilde p'}_t L^{\tilde q'}_{\bar x}L^2_{(x_n,y)}}\Big),
\end{multline}
where
$$
\frac 2p+\frac {n-1}q=1+|\alpha|,\quad 
\frac2{\tilde p}+\frac{n-1}{\tilde q}=\frac{n-1}2,\quad  2<p, \tilde p\leq \infty.
$$ 
\end{prop}
\begin{proof}
It is sufficient to 
consider the case $r=0$. The general case follows by
a derivation with respect to the $y$ variables.
Let 
$$
u(t,\bar x,x_n,y)=e^{{\bf i}t\Delta_{x,y}} f +
\int_0^t e^{{\bf i}(t-\tau)\Delta_{x,y}} F(\tau) d\tau.
$$
Then
$$
{\bf i} \partial_t u+ \Delta_{\bar x} + \partial_{x_n}^2 u + \Delta_y u = F
$$
with
$$u(0,x,y)= f(x, y).$$
Next we introduce  the partial Fourier 
transform of $u, f, F$ with respect to the $x_n$ variable
$$\hat u(t,\bar x, \xi_n, y), \hat f(\bar x, \xi_n, y) 
\hbox{ and }\hat F(t,\bar x, \xi_n, y),$$ 
which satisfy
$${\bf i} \partial_t \hat u+ \Delta_{\bar x}\hat u 
-\xi_n^2 \hat u + \Delta_y \hat u 
= \hat F, \hbox{ } (t,\bar x, y)\in \R_t
\times \R^{n-1}_{\bar x}\times M_y^k$$
with
$$\hat u(0,\bar x,\xi_n, y)= \hat f(\bar x,\xi_n, y).$$
Next, we decompose
$$\hat u(t,\bar x, \xi_n, y), \hat f(\bar x, \xi_n, y) 
\hbox{ and }\hat F(t,\bar x, \xi_n, y) $$ 
with respect to the orthonormal basis $\{\varphi_j\}$ of $L^2(M_y)$
given by the eigenfunctons of $-\Delta_y$ 
(i.e. $-\Delta_y \varphi_j=\lambda_j \varphi_j.$)
Then we have
$$\hat u(t,\bar x, \xi_n, y)=\sum_{j} \hat u_j(t,\bar x, \xi_n) \varphi_j(y)$$
$$\hat F(t,\bar x, \xi_n, y)=\sum_{j} \hat F_j(t,\bar x, \xi_n) \varphi_j(y)$$
$$\hat f(\bar x, \xi_n, y)=\sum_{j} \hat f_j(\bar x, \xi_n) \varphi_j(y).$$
Moreover
$\hat u_j(t,\bar x, \xi_n), \hat f_j(\bar x, \xi_n) 
\hbox{ and }\hat F_j(t,\bar x, \xi_n) $
are related by the following Cauchy problems
\begin{equation}\label{doublefour}
{\bf i} \partial_t \hat u_j+ \Delta_{\bar x}\hat u_j 
- \xi_n^2 \hat u_j - \lambda_j \hat u_j = \hat F_j,
\hbox{ } (t, \bar x, y)\in \R_t\times \R^{n-1}_{\bar x}\times M_y^k
\end{equation}
with
$$\hat u(0,\bar x,\xi_n)= \hat f_j(\bar x,\xi_n).$$
Using Proposition~\ref{above} in the context of \eqref{doublefour} gives
$$ 
\|\partial_{\bar x}^s \hat u_j(t, \bar x, \xi_n)\|_{L^p_tL^q_{\bar x}}\leq 
C\|\hat f_j(\bar x,\xi_n)\|_{H^{\frac{n-3}{2}}_{\bar x}}
+C\sum_{|\beta|=\frac{n-3}{2}} \|\partial^{\beta}_{\bar x} \hat F_j(t, \bar x, \xi_n)
\|_{L^{\tilde p'}_t L^{\tilde q'}_{\bar x}}
$$
where $C=C(p, \tilde p, q, \tilde q)>0$ is constant uniform with respect to $j$ and $\xi_n$
and $p, \tilde p, q, \tilde q$ are as in the assumptions.
In particular we get
$$\|\partial_{\bar x}^s \hat u_j(t, \bar x, \xi_n)
\|_{L^2_{\xi_n}l^2_{j} L^p_tL^q_x}\leq 
C\|f\|_{{\mathcal H}^{\frac{n-3}2,r}_{\bar x,(x_n,y)}}+C\sum_{|\beta|=\frac{n-3}{2}} \|\partial^{\beta}_{\bar x} \hat F_j(t, \bar x, \xi_n)
\|_{L^2_{\xi_n}l^2_{j}L^{\tilde p'}_t L^{\tilde q'}_x}.$$
Again, we use that
$$
\max\{\tilde p', \tilde q'\}\leq 2\leq\min\{p, q\}
$$ 
and therefore the Minkowski inequality gives 
$$
\|\partial_{\bar x}^s \hat u_j(t, \bar x, \xi_n)
\|_{L^p_tL^q_xL^2_{\xi_n}l^2_{j}}
\leq 
C\|f\|_{{\mathcal H}^{\frac{n-3}2,r}_{\bar x,(x_n,y)}}+C\sum_{|\beta|=\frac{n-3}{2}} 
\|\partial^{\beta}_{\bar x} \hat F_j(t, \bar x, \xi_n)\|_{L^{\tilde p'}_t L^{\tilde q'}_x L^2_{\xi_n}l^2_{j}}.
$$
Now, the proof can be concluded by the Plancharel identity (with respect to $x_n$ and $y$) as we did in Proposition~\ref{strichrac}.
\end{proof}
%%%%%%%%%%%%%%
The proof of Theorem \ref{cubicodd} is similar to the proof of Theorem~\ref{cubic} and involves
the following version of Proposition~\ref{multilinear}.
\begin{prop}\label{multil}
Let $n\geq 3$ be odd and $r>\frac{k+1}2$. Then we have the following trilinear estimate
$$
\sum_{|\beta|=\frac{n-3}{2}}
\|\partial_{\bar x}^{\beta}
(1-\partial_{x_n}^2-\Delta_{y})^{r/2}
(u_1 u_2 u_3)\|_{L^{\frac 43}_t L^{\frac{2(n-1)}{n}}_{\bar{x}} L^2_{(x_n,y)}} \leq C 
\|u_1\|_{X}\|u_2\|_{X}\|u_3\|_{X}
$$
where
$$\|u\|_{X}=\sum_{s=0}^{\frac{n-3}2}
\sum_{|\alpha|=s} 
\|\partial_{\bar x}^\alpha(1-\partial_{x_n}^2-\Delta_{y})^{r/2} u
\|_{L^4_t L^{\frac{2(n-1)}{1+2s}}_{\bar{x}}L^2_{(x_n,y)}}.$$
\end{prop}
\begin{proof}
See the proof of Proposition \ref{multilinear}.
\end{proof}
{\bf Proof of Theorem \ref{cubicodd}. }
It is similar to the proof
of Theorem \ref{cubic} provided that
Proposition 
\ref{mixedodd} and \ref{multil}
are used instead of Proposition \ref{mixed}
and \ref{multilinear}.

\end{document}